\documentclass[11pt]{amsart}
\usepackage[top=1.0in,bottom=1.0in,left=1.0in,right=1.0in]{geometry}
\usepackage{graphicx}
\usepackage{amssymb, amsthm,xcolor,mathrsfs, mathtools,stmaryrd}
\usepackage[T1]{fontenc}
\usepackage[utf8]{inputenc}
\usepackage{microtype}
\usepackage{epstopdf}
\usepackage{enumitem} 
\newcommand{\be}{\begin{equation}}
\newcommand{\ee}{\end{equation}}

\newcommand{\bse}{\begin{subequations}}
\newcommand{\ese}{\end{subequations}}

\newcommand{\R}{\mathbb{R}}

\newcommand{\placeholder}{\,\cdot\,}

\newcommand{\domain}{\Omega}

\numberwithin{equation}{section}

\theoremstyle{plain} 
\newtheorem{theorem}{Theorem}[section]

\newtheorem{lemma}[theorem]{Lemma}

\theoremstyle{remark}

\title[Rigidity of internal waves]{Rigidity of three-dimensional internal waves with constant vorticity}

\author[R. M. Chen]{Robin Ming Chen}
\address{Department of Mathematics, University of Pittsburgh, Pittsburgh, PA 15260} 
\email{mingchen@pitt.edu}  

\author[L. Fan]{Lili Fan}
\address{College of Mathematics and Information Science, Henan Normal University, Xinxiang 453007, China} 
\email{fanlily89@126.com} 

\author[S. Walsh]{Samuel Walsh}
\address{Department of Mathematics, University of Missouri, Columbia, MO 65211, USA} 
\email{walshsa@missouri.edu} 

\author[M. H. Wheeler]{Miles H. Wheeler}
\address{Department of Mathematical Sciences, University of Bath, Bath BA2 7AY, United Kingdom}
\email{mw2319@bath.ac.uk}

\date{\today}

\usepackage{hyperref}

\begin{document}

\begin{abstract}
This paper studies the structural implications of constant vorticity for steady three-dimensional internal water waves.  It is known that in many physical regimes, water waves beneath vacuum that have constant vorticity are necessarily two dimensional.  The situation is more subtle for internal waves that traveling along the interface between two immiscible fluids.  When the layers have the same density, there is a large class of explicit steady waves with constant vorticity that are three-dimensional in that the velocity field and pressure depend on one horizontal variable while the interface is an arbitrary function of the other.  We prove the following rigidity result: every three-dimensional traveling internal wave with bounded velocity for which the vorticities in the upper and lower layers are nonzero, constant, and parallel must belong to this family.  If the densities in each layer are distinct, then in fact the flow is fully two dimensional.
\end{abstract}

\maketitle

\section{Introduction}

Depth-varying currents are ubiquitous in the ocean. They can arise from wind-wave interaction, boundary layer effects along the seabed, or tides \cite{miles1957,da1988steep,holthuijsen2010waves}.  Waves riding on currents are essentially rotational, and the interaction of waves with non-uniform currents is described by the vorticity \cite{peregrine1976interaction,jonsson1990wave}.  So far most of the theoretical works on water waves with non-zero vorticity pertains to two-dimensional flows.  The early 19th century work of Gerstner \cite{gerstner1809theorie} furnished a family of exact solutions with a particular nontrivial vorticity distribution that becomes singular at the free surface of the highest wave. Much later, Dubreil-Jacotin \cite{dubreil1934determination} proved the existence of small-amplitude waves with a general vorticity distribution.  After a surge of activity in this area over the last two decades, initiated by Constantin and Strauss \cite{constantin2004exact}, there is now a wealth of small- and large-amplitude existence results for water waves with vorticity; see \cite{haziot2022traveling} for a survey.

Despite these advances in the two-dimensional case, the understanding of three-dimensional rotational waves remains comparatively rudimentary. Currently, there are only two regimes in which existence is known: Lokharu, Seth, and Wahl\'en \cite{lokharu2020existence} have constructed small-amplitude three-dimensional waves with Beltrami-type flow, and Seth, Varholm, and Wahl\'en \cite{seth2022symmetric} obtained symmetric diamond waves with small vorticity. The first result is proved using a careful multi-parameter Lyapunov--Schmidt reduction, while the second involves a delicate fixed-point argument inspired by related problems in plasma physics.

Another body of important recent work concerns the \emph{rigidity} of the governing equations: for certain types of vorticity, the solutions necessarily inherit symmetries of the domain. A number of authors have obtained results of this type for the Euler equations posed in a fixed domain.  Moreover, it is known that finite-depth surface water waves beneath vacuum with non-zero \emph{constant} vorticity are forced to be two dimensional with the vorticity vector pointing in the horizontal direction orthogonal to that of the wave propagation; see \cite{constantin2009effect,constantin2011two,martin2017resonant,stuhlmeier2012constant} for flows beneath surface wave trains and surface solitary waves, \cite{wahlen2014non} for general steady waves, and \cite{martin2018non} for an extension to non-steady waves. Flows with geophysical effects are discussed in the survey article \cite{martin2022three}.

The present paper aims to investigate the structural ramifications of constant vorticity for steady three-dimensional internal water waves. An important feature of waves in the ocean is that the density is heterogeneous due to variations in temperature and salinity.  Commonly, this situation is modeled as two immiscible, superposed layers of constant density fluids.  The interface dividing these regions is a free boundary along which \emph{internal waves} can travel.  Similar to surface waves, the theoretical study on internal waves has been conducted almost exclusively in two dimensions; see \cite[Section 7]{haziot2022traveling}. To the authors' knowledge the only rigorous existence result for genuinely three-dimensional steady internal waves is due to Nilsson \cite{nilsson2019three}, where the flow is assumed to be layer-wise irrotational. It is then natural to ask whether the rigidity of surface water waves with constant vorticity has an internal wave counterpart.  As the latter system has many additional parameters, in principle we might expect it to support a greater variety of flows.  For instance, it can be shown if the vorticity is constant in each layer, then it must be horizontal, but its direction need not be the same in each layer.  On the other hand, if the vorticity vectors are parallel and nonvanishing, we are able to prove a rigidity result that completely characterizes the possible flow patterns. 

\subsection{Formulation}
Consider a three-dimensional traveling wave moving along the interface dividing two immiscible fluids of finite depth and under the influence of gravity.  Fix a Cartesian coordinate system $(x,y,z)$, where $z$ is the vertical direction and the wave propagates in the $xy$-plane.   The fluids are bounded above and below by rigid walls at heights $z = -h_1$ and $z = h_2$, for $h_1, h_2 > 0$.  Adopting a frame of reference moving with the wave renders the system time independent.  Suppose then that the interface between the layers is given by the graph of a $C^1$ function $\eta = \eta(x,y)$.  The fluid domain is thus $\domain := \domain_1 \cup \domain_2$, where the upper layer $\Omega_2$ and lower layer $\Omega_1$ take the form 
\begin{align*}
	\domain_1 &:= \left\{ (x,y,z) \in \mathbb{R}^3 : -h_1 < z <  \eta(x,y) \right\} \\
	\domain_2 &:= \left\{ (x,y,z) \in \mathbb{R}^3 : \eta(x,y) < z < h_2 \right\}.
\end{align*}
See Figure \ref{two-fluid figure} for an illustration.

\begin{figure}
  \centering
  \includegraphics[page=1,scale=1]{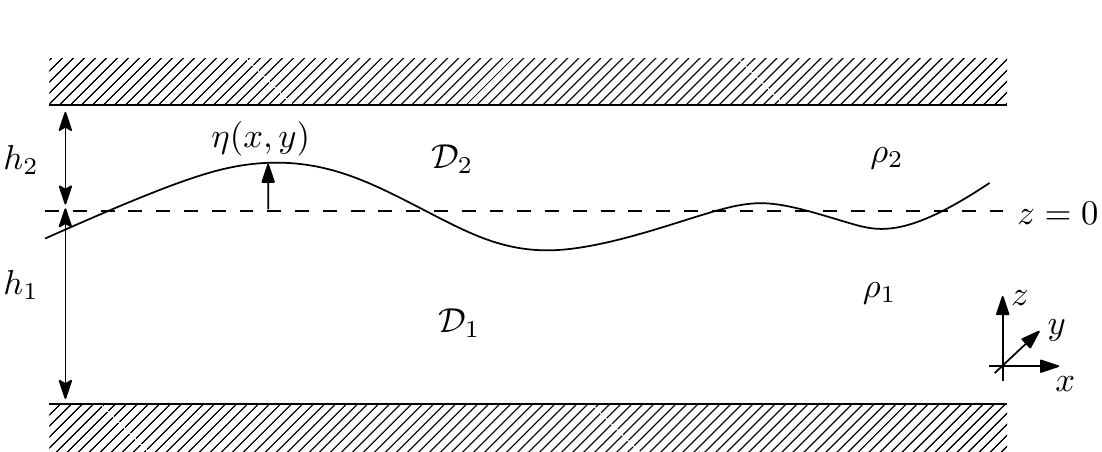}
  \vspace{-2ex}
  \caption{The two-fluid system. 
  \label{two-fluid figure} }  
\end{figure}

For water waves, it is physically reasonable to model the flow in each region as inviscid and incompressible with constant densities $\rho_1, \rho_2 > 0$.  The motion in $\Omega_i$ is described by the (relative) velocity field $\boldsymbol{u}_i := \left( u_i, v_i, w_i \right)$ and pressure $P_i$.   In the bulk, we impose the steady incompressible Euler equations: 
\bse\label{full problem}
\begin{align}
	\rho_i (\boldsymbol{u}_i \cdot \nabla) \boldsymbol{u}_i  & = -\nabla P_i + \rho_i \boldsymbol{g}, \label{momentum equation} \\
	\nabla \cdot \boldsymbol{u}_i & = 0, \label{incompressibility}
\end{align}
where $\boldsymbol{g} := (0,0,-g)$ is the (constant) gravitational acceleration vector.
The first of these mandates the conservation of momentum, while the second is the incompressibility condition.  The boundary conditions at the interface are the continuity of normal velocity and pressure:
\begin{align}
u_i \eta_x + v_i \eta_y & = w_i \quad \text{on } z = \eta(x, y), \label{kinematic on interface} \\
P_1  & = P_2   \quad  \text{on } z = \eta(x, y). \label{dynamic on interface}
\end{align}
   On the upper and lower rigid boundaries, the kinematic boundary conditions are
\be\label{kinematic on rigid}
\begin{aligned}
	w_1 & = 0 & \qquad & \textrm{on } z = -h_1 \\
	w_2 & = 0 & \qquad & \textrm{on } z = h_2.
\end{aligned}
\ee
\ese
These say simply that the velocity field is tangential to the rigid walls. Throughout this paper, we consider classical solutions for which $\boldsymbol{u}_i \in C^1(\overline{\domain_i}; \mathbb{R}^3)$, $P_i \in C^1(\overline{\domain_i})$, and $\eta \in C^1(\mathbb{R}^2)$.  In order to ensure there is a positive separation between the interface and the walls, we further assume that $-h_1 < \inf{\eta}$ and $\sup{\eta} < h_2$.  

Recall that the \emph{vorticity} in the layer $\domain_i$ is defined to be the vector field
\be\label{def vorticity}
	\boldsymbol{\omega}_i := \nabla \times \boldsymbol{u}_i = (\partial_y w_{i} - \partial_z v_{i},\,  \partial_z u_{i} - \partial_x w_{i},\, \partial_x v_{i} - \partial_y u_{i}).  
\ee 
Taking the curl of the momentum equation \eqref{momentum equation}, we find that each $\boldsymbol{\omega}_i$ satisfies the so-called steady vorticity equation 
\be\label{full vorticity eqn}
	(\boldsymbol{u}_i \cdot \nabla) \boldsymbol{\omega}_i = (\boldsymbol{\omega}_i \cdot \nabla) \boldsymbol{u}_i \qquad \textrm{in } \domain_i.
\ee
Suppose now that the vorticity in each layer is a nonzero constant
\be\label{const vorticity}
	\boldsymbol{\omega}_i = (\alpha_i, \beta_i, \gamma_i) \qquad \textrm{for } i = 1,2.
\ee
Then the advection term on the left-hand side of \eqref{full vorticity eqn} vanishes identically, while the vortex stretching term on the right-hand side becomes a constant directional derivative of $\boldsymbol{u}$:
\be
\label{vorticity eq}
	(\boldsymbol{\omega}_i \cdot \nabla) \boldsymbol{u}_i = 0 \qquad \textrm{in } \domain_i.
\ee
Thus, the velocity $\boldsymbol{u}_i$ is constant in the direction of $\boldsymbol{\omega}_i$.  As \eqref{full problem} is invariant under rotation about the $z$-axis, we can without loss of generality assume that $\boldsymbol{\omega}_2 = (0, \beta_2, \gamma_2)$, that is, the vorticity of the upper fluid lies in the $yz$-plane.  

From the vector identity
\[
(\boldsymbol{u} \cdot \nabla) \boldsymbol{u} + \boldsymbol{u} \times \boldsymbol{\omega} = (\nabla \cdot \boldsymbol{u}) \boldsymbol{u} + \frac12 \nabla (|\boldsymbol{u}|^2),
\]
one can rewrite \eqref{momentum equation} as 
\be\label{B eqn}
\boldsymbol{u}_i \times \boldsymbol{\omega}_i = \nabla H_i
\ee
where 
\be\label{Bernoulli quantity}
H_i := \frac12 |\boldsymbol{u}_i|^2 + \frac{P_i}{\rho_i} + g z
\ee
is called the Bernoulli function. From \eqref{B eqn} we  see that $H_i$ is constant along the vortex lines.

\subsection{Main results}

Our first theorem imposes a dimensionality constraint on the vorticity: if $\boldsymbol{u}_1$ and $\boldsymbol{u}_2$ are uniformly bounded, then both vorticity vectors $\boldsymbol{\omega}_1$ and $\boldsymbol{\omega}_2$ are necessarily two-dimensional and lie in the $xy$-plane.  A result of this type was first proved by Wahl\'en \cite{wahlen2014non} for steady gravity and capillary-gravity water waves beneath vacuum.  Martin \cite{martin2018non} later showed the same holds for the time-dependent case.  Adapting Wahl\'en's argument to the two-layer case requires some nontrivial new analysis due to the more complicated behavior at the interface.  Ultimately, we obtain the following.

\begin{theorem}\label{thm vorticity}
\label{vorticity theorem}
Consider a solution to the internal wave problem \eqref{full problem} such that $\|\boldsymbol{u}_1\|_{C^0}, \| \boldsymbol{u}_2 \|_{C^0} < \infty$ and $\boldsymbol{\omega}_1$ and $\boldsymbol{\omega}_2$ are nonzero constant vectors.  Then necessarily the third components of $\boldsymbol{\omega}_1$ and $\boldsymbol{\omega}_2$ both vanish.
\end{theorem}

Next, we consider the structure of the velocity field and free surface profile.  Under remarkably general conditions, Wahl\'en \cite{wahlen2014non} proves that for a gravity wave beneath vacuum, if the vorticity is constant, then the flow must be entirely two dimensional: $\boldsymbol{u}_i$ lies in the $xz$-plane and depends only on $(x,z)$, while $\eta = \eta(x)$.  In other words, genuinely three-dimensional steady surface gravity water waves with non-zero constant vorticity do not exist.  Wahl\'en also proves the same holds for capillary-gravity waves provided the velocity field and free surface profile are uniformly bounded in $C^1$, and a Taylor sign condition on the pressure holds.  Earlier work by Constantin \cite{constantin2011two}, Constantin and Kartashova \cite{constantin2009effect}, and Martin \cite{martin2017resonant} obtain analogous results for gravity and capillary-gravity waves under the more restrictive assumption that $\eta$ is periodic, while Martin treats time-dependent \cite{martin2018non} and viscous waves \cite{martin2022flow} again with a Taylor sign condition; see also the survey in \cite{martin2022three}.  The moral of this body of work is that in order to find genuinely three-dimensional steady rotational waves beneath vacuum, one must allow for a more complicated vorticity distribution.  

However, constant vorticity internal waves are not obliged to be two dimensional.  Indeed, a little thought readily leads us to a profusion of explicit three-dimensional solutions to \eqref{full problem}.   Assume that we are in the the Boussinesq setting where $\rho_1 = \rho_2$.  Then, taking 
\be \label{trivial shear}
	\boldsymbol{u}_i  = (\beta_i z + k_i, 0, 0) \quad \textrm{for } i = 1,2, \qquad
	\eta = H(y)
\ee
gives a steady wave for any $H \in C^1(\mathbb{R}; (-h_1,h_2))$. Note that the corresponding vorticity vectors $\boldsymbol{\omega}_i = (0,\beta_i, 0)$ are parallel.  We can visualize \eqref{trivial shear} as two shear flows defined in $\overline{\domain}$, which when $\rho_1 = \rho_2$ will have the same (hydrostatic) pressure.  Any streamline in the $xz$-plane can be viewed as a material interface above which we have the first fluid and below the second.  When $v_1, v_2 \equiv 0$, we can smoothly vary which streamline is the interface as we change $y$, permitting there to be three-dimensional structure.  Essentially, the difference between the situation here and the one-layer case lies in the dynamic condition \eqref{dynamic on interface}.  When the fluid is bounded above by vacuum, the pressure must be constant along the interface, whereas for internal waves it need only be continuous.

Members of the family of solutions \eqref{trivial shear} can be thought of as \emph{trivially three-dimensional shear flows} when $\eta_y \not\equiv 0$.  Our main theorem shows that they are in fact the only possible configuration for three-dimensional waves with constant parallel vorticity and bounded velocity. 

\begin{theorem}[Rigidity]
\label{rigidity theorem}
Every solution of the internal wave problem \eqref{full problem} for which 
\begin{enumerate}[label=\rm(\roman*)]
\item  $\boldsymbol{\omega}_1$ and $\boldsymbol{\omega}_2$ are constant, parallel, and nonzero, and
\item  $\| \boldsymbol{u}_1 \|_{C^0}$, $\| \boldsymbol{u}_2 \|_{C^0} < \infty$,
\end{enumerate}
is either a trivially three-dimensional shear flow of the form \eqref{trivial shear} or two dimensional.  If $\rho_1 \neq \rho_2$, then the wave is necessarily two dimensional.
\end{theorem}

One can interpret this theorem as the statement that the solutions to \eqref{full problem} inherit the symmetry of the channel domain in which the problem is posed.  Related rigidity results for the two-dimensional Euler equations have been obtained by Hamel and Nadirashvili \cite{hamel2017shear,hamel2019liouville}, who prove that all solutions in a strip, half plane, or the whole plane with no stagnation points 
are shear flows (that is, the vertical velocity vanishes identically and the horizontal velocity depends only on $z$). Under the same no-stagnation assumption, these authors also find that steady Euler configurations confined to circular domains must be radially symmetric \cite{hamel2021circular}. Allowing the presence of stagnation points, G\'omez-Serrano, Park, Shi, and Yao \cite{gomez2021symmetry} show that smooth stationary solution with compactly supported and nonnegative vorticity must be radial. In Theorem~\ref{rigidity theorem} we avoid making any restrictions on the velocity beyond boundedness, though we only treat the constant vorticity case.  Notably, as in \cite{wahlen2014non}, we make no a priori assumptions on the far-field behavior of the wave.  Thus in the non-Boussinesq case, nontrivial solitary waves, periodic waves, fronts --- and all other more exotic waveforms --- are excluded all at once.  We also mention that it is possible to rule out capillary-gravity internal waves through arguments similar to the one-fluid regime; see Theorem~\ref{capillary theorem}. 

The idea of the proof can be explained as follows.  Thanks to Theorem~\ref{thm vorticity}, when $\boldsymbol{\omega}_1$ and $\boldsymbol{\omega}_2$ are parallel, the velocity fields are two-dimensional: $\boldsymbol{u}_i = (u_i(x,z), V_i, w_i(x,z))$ where $V_i$ are constants.  The same also holds for the pressures, but a priori $\eta$ may depend on both $(x,y)$.  If the interface is not independent of $y$, then the projections $\tilde \domain_i$ of $\domain_i$ into the $xz$-plane will have non-empty intersection with non-empty interior, and on that set we have two solutions of the two-dimensional Euler equations.  Because each point in $\tilde \domain_1 \cap \tilde \domain_2$ corresponds to one or more points on the interface, the dynamic condition applies throughout.  The key insight of Wahl\'en is that, for waves beneath vacuum, this forces the pressure to be constant, and hence by analyticity, it is constant throughout the fluid.  As this is not possible, he concludes that for surface waves, the interface must be flat in $y$.  For internal waves, however, the dynamic condition tells us instead that there exists a pressure $P = P(x,z)$ that is real analytic on $\tilde \domain_1 \cup \tilde \domain_2$ and whose restriction to $\tilde\domain_1$ is $P_1$ and whose restriction to $\tilde \domain_2$ is $P_2$.

The central question therefore turns to one of uniqueness of steady solutions of the two-dimensional Euler equations with a prescribed pressure, but allowing for potentially different densities and different constant vorticities.  We have in addition that the kinematic condition \eqref{kinematic on interface} holds on the intersection region, which forces a relation between the slopes of the two velocity fields there.  Through a novel but elementary argument, we prove that the streamlines (integral curves) of the vector fields $(u_1, w_1)$ and $(u_2, w_2)$ coincide on $\tilde \Omega_1 \cap \tilde \Omega_2$.  Finally, from the real analyticity of the velocity and pressure and Liouville's theorem, we are ultimately able to conclude that the pressure must be hydrostatic, and thus the wave is of the form \eqref{trivial shear}.  

\section{Dimension reduction for the vorticity}

This section is devoted to the proof of Theorem~\ref{vorticity theorem} on the two-dimensionality of the vorticity.  As mentioned above, it is based on the corresponding work of Wahl\'en in \cite{wahlen2014non}.    The main point is that this argument relies on the structure of the velocity field near the rigid walls, which is not substantially different in the two-fluid setting. 

A key observation, both for the present section and the next, is that each component of the velocity is harmonic:
\be
\label{harmonic velocity}
	\Delta u_i = \Delta v_i = \Delta w_i = 0 \qquad \textrm{in } \domain_i.
\ee
This follows simply by taking the curl of equation \eqref{def vorticity} and using incompressibility \eqref{incompressibility}.  As just one important consequence, $u_i$, $v_i$, and $w_i$ are all real-analytic functions.  Taking the divergence of the momentum equation \eqref{momentum equation}, we likewise find that the pressure $P_i$ solves a Poisson equation with real-analytic forcing, and hence it too is real analytic.  These facts will be crucial to our analysis at several points.  In particular, they provide a means to globalize identities that hold on open subsets to the entirety of the fluid domain.

\begin{proof}[Proof of Theorem~\ref{thm vorticity}]
Seeking a contradiction, suppose that one of $\gamma_i$ is not zero, say, $\gamma_1 \ne 0$; the argument for the other case $\gamma_2 \ne 0$ can be treated the same way. Then from the third component of the vorticity equation \eqref{vorticity eq} we see that $w_1$ is constant in the direction of $\boldsymbol{\omega}_1$, which is transverse to the lower boundary at $ z = - h_1$. From the kinematic condition \eqref{kinematic on rigid}, it follows that $w_1$ vanishes identically on the open neighborhood $\mathcal{N} := \{ (x,y,z) : -h_1 < z < \inf\eta \}$ of the bed. As it is real analytic, this forces
\[
	w_1 \equiv 0 \quad \text{in }\ \domain_1.
\]
Reconciling this with \eqref{const vorticity}, \eqref{incompressibility} and \eqref{momentum equation}, we then have 
\begin{align}
& \partial_z u_{1} = \beta_1, \quad \partial_z v_{1} = -\alpha_1,\label{reduced uv_z} \\
& \partial_x u_{1} + \partial_y v_{1}  = 0, \label{reduced incomp}\\
& \partial_z P_{1} = -\rho_1 g \label{reduced P_z}
\end{align}
in $\domain_1$.  By integrating \eqref{reduced uv_z}, we infer that
\be\label{reduced uv}
u_1 = \bar u_1(x, y) + \beta_1 z, \qquad v_1 = \bar v_1(x, y) - \alpha_1 z,
\ee
in $\mathcal{N}$ for some functions $\bar u_1$ and $\bar v_1$. The reduced incompressibility condition \eqref{reduced incomp} then implies that 
\[
	\partial_x \bar u_{1} + \partial_y \bar v_{1} = 0,
\]
  which ensures the existence of a reduced stream function $\bar \psi_1 = \bar \psi_1(x, y)$ defined on $\mathcal{N}$ such that $\nabla^\perp \bar\psi_1 = (-\partial_z \bar\psi,\partial_x \bar\psi) = (\bar u_1, \bar v_1)$.  From the horizontal momentum equations \eqref{momentum equation} and the form of the pressure obtained in \eqref{reduced P_z}, we see in $\mathcal{N}$, $\bar \psi_1$ satisfies
\be\label{reduced momentum eqn}
\left\{ 
\begin{split}
	\beta_1 \partial_x \partial_y \bar\psi_{1} - \alpha_1 \partial_y^2 \bar\psi_{1} & = 0 \\
	-\beta_1 \partial_x^2 \bar\psi_{1} + \alpha_1 \partial_x \partial_y\bar\psi_{1} & = 0 \\
	\Delta \bar\psi_1 + \gamma_1 & = 0
\end{split}
\right.
\qquad \textrm{in } \mathcal{N}.
\ee
We consider two cases.

{\bf Case 1: $\alpha_1^2 + \beta_1^2 = 0$.} From \eqref{reduced uv_z} and \eqref{const vorticity} it follows that
\be\label{vanishing uv_z}
\partial_z u_{1} = \partial_z v_{1} = 0, \qquad \partial_x v_{1} - \partial_y u_{1} = \gamma_1.
\ee
We also find from \eqref{reduced uv} and \eqref{harmonic velocity} that in the neighborhood $\mathcal{N}$,  $u_1 = \bar u_1$ and $v_1 = \bar v_1$ are harmonic functions with domain $\R^2$. The boundedness of $\boldsymbol{u}_1$, and thus the boundedness of $(\bar u_1, \bar v_1)$, allows one to appeal to the Liouville theorem for harmonic functions to conclude that $u_1$ and $v_1$ are constants. However this contradicts that fact that $\gamma_1 \ne 0$.

{\bf Case 2: $\alpha_1^2 + \beta_1^2 \ne 0$.} In this case, direct computation from \eqref{reduced momentum eqn} yields that the second-order derivatives of $\bar\psi_1$ are all constant:
\be\label{soln psi}
	\partial_x^2 \bar\psi_{1} = -{\alpha_1^2 \gamma_1 \over \alpha_1^2 + \beta_1^2} =: A_1, \quad \partial_x \partial_y \bar\psi_{1} = -{\alpha_1\beta_1\gamma_1 \over \alpha_1^2 + \beta_1^2} =: B_1, \quad \partial_y^2 \bar\psi_{1} = -{\beta_1^2\gamma_1 \over \alpha_1^2 + \beta_1^2} =: C_1, 
\ee
from which one can solve for $\bar u_1$ and $\bar v_1$
\[
	\bar u_1 = B_1 x + C_1 y + a_1, \qquad \bar v_1 = -A_1 x - B_1 y + b_1
\]
for some constants $a_1$ and $b_1$. Thus
\be\label{soln uv}
u_1 = B_1 x + C_1 y + \beta_1 z  + a_1, \qquad v_1 = -A_1 x - B_1 y - \alpha_1 z + b_1.
\ee
Again boundedness of $\boldsymbol{u}_1$ forces $A_1 = B_1 = C_1 = 0$, leading to $\alpha_1 = \beta_1 = 0$, a contradiction.
\end{proof}

\section{Rigidity of internal waves}

The purpose of this section is to prove the rigidity result in Theorem~\ref{rigidity theorem}, characterizing three-dimensional internal waves with constant vorticity.  Recall that we have, without loss of generality, chosen axes so that $\boldsymbol{\omega}_2$ lies in the $xz$-plane.   Theorem \ref{thm vorticity} then guarantees that the vorticity in each layer takes the form
\be\label{vorticity form}
\boldsymbol{\omega}_1 = (\alpha_1, \beta_1, 0), \qquad \boldsymbol{\omega}_2 = (0, \beta_2, 0).
\ee
Note that the assumption $\boldsymbol{\omega}_1$ and $\boldsymbol{\omega}_2$ are parallel is equivalent to $\alpha_1 = 0$.  More generally, though, the particularly simple form of $\boldsymbol{\omega}_2$ allows one us further characterize the flow pattern in the upper layer. 
\begin{lemma}\label{lem y indep}
Let the assumptions of Theorem \ref{thm vorticity} hold. Then, $\boldsymbol{u}_2$ and $P_2$ are independent of $y$, and $v_2$ is constant.  Likewise, $\boldsymbol{u}_1$ and $P_1$ are constant along lines parallel to $\boldsymbol\omega_1$, while $\alpha_1 u_1 + \beta_1 v_1$ is constant.
\end{lemma}
\begin{proof}
We will only present the argument for the upper fluid as the lower fluid follows through essentially the same reasoning.  From \eqref{vorticity form}, \eqref{const vorticity} and \eqref{vorticity eq} it follows that
\[
	\partial_y u_{2} = \partial_y v_{2} = \partial_y w_{2} = 0, \qquad \partial_x v_{2} = \partial_z v_{2} = 0.
\]
In particular, $\nabla v_2 = 0$, and thus $v_2$ is a constant throughout $\domain_2$. The momentum equation in $y$-directional momentum equation then becomes
\[
	\partial_y P_{2} = 0 \qquad \textrm{in } \domain_2.
\]
Following the argument as in \cite[Lemma 3]{wahlen2014non} using the real analyticity of $P_2$ we can show that $P_2$ is independent of $y$ in the upper fluid layer $\domain_2$. In fact we see that $P_2$ is independent of $y$ in a region sufficiently close to the top boundary $\{ z = h_2 \}$. Therefore for any $y_1 \ne y_2$, there exists a minimal $z_* \le h_2$ such that  
\[
P_2(x, y_1, z) = P_2(x, y_2, z) \quad \text{for} \quad z_* \le z \le h_2.
\]
Clearly we know that $z_* \ge \max\{ \eta(x,y_1), \eta(x,y_2) \}$. Using the real analyticity of $z \mapsto P_2(x, y_1, z) - P_2(x, y_2, z)$ we see that $z_* = \max\{ \eta(x,y_1), \eta(x,y_2) \}$, which indicates that $P_2$ is independent of $y$ in $\domain_2$. The result for $(u_2, w_2)$ follows in a similar way.
\end{proof}

Let us now proceed to the proof of the main result. 

\begin{proof}[Proof of Theorem~\ref{rigidity theorem}]
Thanks to Lemma~\ref{lem y indep} and the assumption that $\boldsymbol{\omega}_1$ and $\boldsymbol{\omega}_2$ are parallel and non-vanishing, we have that $v_1$ and $v_2$ are constants; let them be denoted $V_1$ and $V_2$, respectively. Moreover, $\boldsymbol u_1, P_1, \boldsymbol u_2$, and $P_2$ are independent of $y$, so we  can write 
\[
	\begin{aligned}
		\boldsymbol{u}_1(x,y,z) & = \tilde{\boldsymbol{u}}_1(x,z), & \quad P_1(x, y, z) & = \tilde P_1(x,z), \\
		\boldsymbol{u}_2(x,y,z) & = \tilde{\boldsymbol{u}}_2(x,z), & \quad P_2(x, y, z) & = \tilde P_2(x,z),
	\end{aligned}
\]
where $\tilde{\boldsymbol{u}}_i$ and $\tilde P_i$ are defined on the projection 
\be\label{proj domain}
\tilde\Omega_i := \left\{ (x, z):\ (x, y, z) \in \Omega_i \text{ for some } y \in \R \right\},
\ee
of $\Omega_i$ on the $xz$-plane, for $i = 1, 2$. It is easy to see that in fact
\be\label{proj domain exp}
\tilde\Omega_1 = \left\{ (x, z):\ -h_1 < z < f_1(x) \right\}, \quad \tilde\Omega_2 = \left\{ (x,z):\ f_2(x) < z < h_2 \right\},
\ee
where
\[
f_1(x) := \sup_{y\in \R} \eta(x,y), \quad \text{and} \quad f_2(x) := \inf_{y\in \R} \eta(x,y).
\]
By definition $f_2(x) \le f_1(x)$. The boundedness of $\eta$ implies that $-h_1 < f_1(x) \le h_2$ and $-h_1 \le f_2(x) < h_2$. It is elementary that  $f_1$ is then lower semicontinuous while $f_2$ is upper semicontinuous. The projected planes $\tilde\Omega_i$ are both open and connected subsets of $\R^2$, for $i = 1,2$. 

Arguing by contrapositive, suppose that $\eta_y \not\equiv 0$.  Then $\tilde\Omega_1 \cap \tilde\Omega_2 \ne \emptyset$ and there exists some point $(x_0, y_0)$ such that $z_0 := \eta(x_0, y_0) \in (f_2(x_0), f_1(x_0))$. The dynamic boundary condition \eqref{dynamic on interface} yields
\[
\tilde P_1(x_0, z_0) = \tilde P_2(x_0, z_0).
\]
A continuity argument implies that for each $z$ between $z_0$ and $f_1(x_0)$ there exists some $y(z)$ such that $z = \eta(x_0, y(z))$. Therefore on the line segment joining $(x_0, z_0)$ and $(x_0, f_1(x_0))$ we have
\[
\tilde P_1(x_0, z) = \tilde P_2(x_0, z).
\]
See Figure \ref{projection figure}. 
\begin{figure}
  \centering
  \includegraphics[page=2,scale=1]{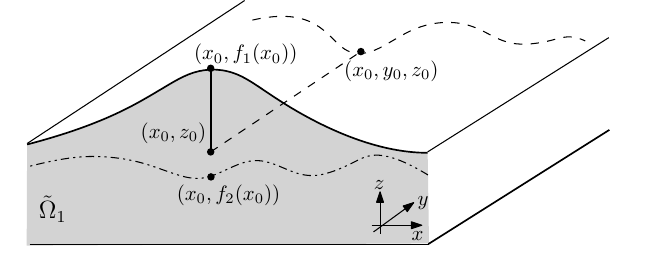}
  \vspace{-2ex}
  \caption{Projection to $\tilde\Omega_1$. 
  \label{projection figure} }  
\end{figure}
Now from the lower semicontinuity of $f_1$ and the upper semicontinuity of $f_2$ we know that for $x$ sufficiently close to $x_0$ it holds that $f_2(x) < \eta(x, y_0) < f_1(x)$. Repeating the previous argument it follows that there exists an open subset of $\tilde\Omega_1$ in which $\tilde P_1(x, z) = \tilde P_2(x, z)$. The analyticity of $\tilde P_i$ then forces $\tilde P_1 = \tilde P_2$ on $\tilde\Omega_1 \cap \tilde\Omega_2$, and thus $\tilde P_1$ and $\tilde P_2$ are analytic extensions of each other in the entire strip $\tilde\Omega := \left\{ (x,z):\ -h_1 < z < h_2 \right\}$.

Recall that we say the pressure in $\tilde\Omega_i$ is \emph{hydrostatic} provided $\nabla (\tilde P_i + \rho_i z)$ vanishes identically. Suppose that either $\tilde P_1$ or $\tilde P_2$ is hydrostatic. Then uniqueness of the analytic extension implies both are hydrostatic and hence $\rho_1 = \rho_2$. The incompressibility of $\tilde{\boldsymbol{u}}_1$ and $\tilde{\boldsymbol{u}}_2$ permit us to define stream functions $\tilde \psi_i$ by $\nabla^\perp \tilde \psi_i := (\tilde u_i, \tilde w_i)$. The Bernoulli equation \eqref{B eqn} now reads
\[
\beta_i(-\tilde w_i, 0, \tilde u_i) = \nabla \left[ \frac{1}{2}({\tilde u_i}^2 + V_i^2 + {\tilde w_i}^2) \right] \quad \textrm{in } \Omega_i,
\]
which in turn leads to
\[
\frac{1}{2} \left( \tilde u_i^2 + \tilde w_i^2\right) - \beta_i \tilde \psi_i = Q_i \quad \textrm{in } \tilde\Omega_i
\]
for some constant $Q_i$.  On the bed, $\tilde w_1 \equiv 0$ and $\tilde \psi_1$ is constant, and so $\tilde u_1$ is likewise constant there.  Thus, 
\[
	\tilde w_{1z} = \tilde u_{1x} = 0 \qquad \textrm{on } \{ z = -h_1 \},
\]
and, because $\tilde w$ is harmonic, it must therefore be that $\tilde w_1 \equiv 0$ in $\tilde\Omega_1$.  The same argument applied on the lid shows that $\tilde w_2 \equiv 0$ in $\tilde\Omega_2$.  Incompressibility then implies that $\tilde u_1 = U_1(z)$ and $\tilde u_2 = U_2(z)$, meaning we have a shear flow. The constant vorticity then forces $\tilde u_i = \beta_i z + k_i$ as in \eqref{trivial shear}.

Evaluating the kinematic condition using this fact gives
\[
	(\beta_1 \eta + k_1) \eta_x + V_1 \eta_y = 0 \qquad \textrm{for all } (x,y) \in \mathbb{R}^2.
\]
If $V_1 \neq 0$, this is Burgers' equation with $y$ playing the role of the evolution variable.   Because the only global classical solutions are constants, this forces the interface to be perfectly flat.  On the other hand, if $V_1 = 0$, we can simply integrate the equation in $x$ to see that $\eta(\placeholder, y)$ is likewise constant.  In either case, then, the wave is completely shear with no variation in the $x$-direction.

As the converse of these inferences is obviously true, the conclusions of the previous two paragraphs can be stated succinctly as:
\begin{equation}
  \label{hydrostatic iff trivial}
  \text{$\tilde P_1$ or $\tilde P_2$ hydrostatic} ~ \Longleftrightarrow ~  \text{$\tilde P_1$ and  $\tilde P_2$ hydrostatic} ~ \Longleftrightarrow ~ \left\{ \begin{aligned}  
    \tilde u_1 & = \beta_1 z + k_1, & \tilde w_1 \equiv 0, \\
    \tilde u_2 & = \beta_2 z + k_2, & \tilde w_2 \equiv 0, \\
     \eta & = H(y), & \\
    \rho_1 & = \rho_2. & \end{aligned} \right.
\end{equation}
Our goal in the remainder of the proof is therefore to show that at least one of $\tilde P_1$ and $\tilde P_2$ is hydrostatic.

The  kinematic condition in the projected domain states that
\[
	\left\{
	\begin{aligned}
		\tilde u_1(x, z) \eta_x(x, y) + V_1 \eta_y(x,y) & = \tilde w_1(x, z) \\
		\tilde u_2(x, z) \eta_x(x, y) + V_2 \eta_y(x,y)  & = \tilde w_2(x, z),
	\end{aligned}
	\right.
\]
where $(x,z) \in \tilde\Omega_1 \cap \tilde\Omega_2$ and $y$ is any point such that $z = \eta(x, y)$.  Observe that this can be rewritten in terms of the stream functions as
\be \label{kinematic system psi}
		\left\{
	\begin{aligned}
		V_1 \eta_y(x,y) & = \partial_x \left( \tilde \psi_1(x,\eta(x,y)) \right) \\
		V_2 \eta_y(x,y) & = \partial_x \left( \tilde \psi_2(x,\eta(x,y)) \right).
	\end{aligned}
	\right.
\ee

Let us look at two possibilities.  First suppose that $V_1 = V_2 = 0$.  Thus from \eqref{kinematic system psi}, we see that each graph $\eta(\placeholder, y)$ is a streamline for both $\tilde{\boldsymbol{u}}_1$ and $\tilde{\boldsymbol{u}}_2$. It follows that the Poisson bracket of $\tilde\psi_1$ and $\tilde \psi_2$ vanishes identically in $\tilde \Omega_1 \cap \tilde \Omega_2$.  By real analyticity, the zero-set of $|\nabla \tilde \psi_1|^2 |\nabla \tilde \psi_2|^2$ is either the entirety of $\tilde \Omega_1\cap \tilde \Omega_2$ or a closed, nowhere dense subset.  In the first case, we would of course have that the flow is hydrostatic, so assume that the latter is true.  Then we can find an open set $\mathcal{U} \subset \tilde \Omega_1 \cap \tilde \Omega_2$ on which $|\nabla \tilde \psi_1|, |\nabla \tilde \psi_2| \neq 0$.  It follows that there exists some real-analytic function $\Lambda$ such that
$\psi_1 = \Lambda(\psi_2)$ on $\mathcal{U}$.  Taking the Laplacian of both sides then gives the identity
\[
	\beta_1 = \Lambda^{\prime\prime}(\tilde\psi_2) |\nabla \tilde\psi_2|^2  + \Lambda^\prime(\tilde \psi_2) \beta_2 \qquad \textrm{on } \mathcal{U}.
\]

We see then that either $\Lambda^{\prime\prime}(\tilde\psi_2) \equiv 0$, or else $|\nabla \tilde \psi_2|^2$ is constant along the streamlines in some open subset $\mathcal{V} \subset \mathcal{U}$.  In the first case, $\lambda := \Lambda^\prime(\tilde \psi_2)$ is constant on $\mathcal{V}$, and so by real analyticity, $(\tilde u_1, \tilde w_1) = \lambda (\tilde u_2, \tilde w_2)$ on all of $\mathcal{U}$.  We can thus extend $\tilde u_1$ and $\tilde w_1$ as real-analytic (indeed, harmonic) functions defined on the entire closure of $\tilde\Omega$ with $\tilde u_1 = \lambda \tilde u_2$ and $\tilde w_1 = \lambda \tilde w_2$ on $\tilde\Omega_2$.   The Phragm\'en--Lindel\"of principle and boundary conditions then force $\tilde w_1 \equiv 0$, so by incompressibility $\tilde u_{1x} \equiv 0$.  Thus $\tilde P_1$ is hydrostatic, and we can appeal to \eqref{hydrostatic iff trivial} to show that the wave is trivial.   

Assume next that $|\nabla \tilde \psi_2|^2$ is constant along the streamlines in $\mathcal{V}$.  Bernoulli's law then implies that the dynamic pressure $p_2 := \tilde P_2 - \rho_2 g z$ is also constant along the streamlines in $\mathcal{V}$, that is, $\nabla p_2 \cdot \nabla^\perp \tilde\psi_2 = 0$ in $\mathcal{V}$. By construction, $\nabla^\perp \tilde\psi_2 = (\tilde u_2, \tilde w_2)$ has no stagnation points in $\tilde\Omega_2$. So by analyticity we have $\nabla p_2 \cdot \nabla^\perp \tilde\psi_2 = 0$ in $\tilde\Omega_2$. In particular, $p_2$, and thus $\tilde P_2$, is constant on $z = h_2$, which by the argument above forces $\tilde P_2$ to be hydrostatic. 

Next consider the situation where at least one of $V_1$ and $V_2$ is non-vanishing; for definiteness, say $V_1 \neq 0$.  Unlike the previous case, the graphs of $\eta(\placeholder, y)$ are no longer streamlines, however \eqref{kinematic system psi} implies that for any $y \in \mathbb{R}$, 
\[
	V_2 \tilde \psi_1 - V_1 \tilde \psi_2 \quad \textrm{is constant on } \{ \eta(x,y) : x \in \mathbb{R} \}.
\]
As we have assumed $\eta_y \not\equiv 0$, we may let $(x_0,z_0) \in \tilde \Omega_1 \cap \tilde \Omega_2$ be given such that $z = \eta(x_0, y_0)$ and $\eta_y(x_0, y_0) \neq 0$.  Let $(a,b)$ be an open interval containing $y_0$ on which $\eta(x_0, \placeholder)$ is monotone.  Integrating the kinematic condition \eqref{kinematic system psi} from $x = x_0$ to $x = M$ and from $y = a$ to $y = b$ gives
\begin{align*}
	V_1 \int_{x_0}^M \left( \eta(x,b) - \eta(x, a) \right) \, dx & = \int_{x_0}^M  \int_a^b \partial_x \left( \tilde \psi_1(x, \eta(x,y)) \right) \, dy \, dx \\
		& = \int_a^b \tilde \psi_1(M, \eta(M, y)) \, dy - \int_a^b \tilde \psi_1(x_0, \eta(x_0, y)) \, dy.
\end{align*}
The right-hand side above is bounded uniformly in $M$ since
\[
	\left| \int_a^b \tilde \psi_1(M, \eta(M, y)) \, dy \right|\leq (b-a) \| \tilde \psi_1 \|_{C^0} \lesssim \| \tilde u_1 \|_{C^0}.
\]
Therefore, we must have that $\inf_{x \geq x_0} |\eta(x,b)-\eta(x,a)| = 0$, as otherwise, the left-hand side integral would diverge as $M \to \infty$.  That is, the distance between the graphs $\eta(\placeholder, y_1)$ and $\eta(\placeholder, y_2)$ is in fact $0$ for all for $y_1,y_2 \in (a,b)$.  It follows that $V_2 \tilde \psi_1 - V_1 \tilde \psi_2$ is constant in the set $\mathcal{W}$ that is bounded above and below by the graphs of $\eta(\placeholder, b)$ and $\eta(\placeholder, a)$.  But since $\eta_y(x_0,z_0) \neq 0$, the inverse function theorem applied to $(x,y) \mapsto (x, \eta(x,y))$ ensures that some open neighborhood $\mathcal{U} \ni (x_0,z_0)$ lies in the interior of $\mathcal{W}$. 

From here, it is easy to see that the flow must be hydrostatic.  If $V_2 = 0$, by analyticity we would have that $(\tilde u_1, \tilde w_1) \equiv (0,0)$, meaning $\beta_1 = 0$ and the flow is hydrostatic.  If $V_2 \neq 0$, then we can write $\psi_1 = \Lambda(\psi_2)$ for an affine function $\Lambda$.  The argument from the previous case shows that this forces the pressure to be hydrostatic.  
\end{proof}

\section{Discussion}

We conclude the paper with some informal discussion of some simple extensions, as well as two open problems stemming from the arguments above.  

\subsection*{Capillary-gravity internal waves}

One can also consider the question of rigidity for \emph{capillary-gravity} internal waves, meaning the effects of surface tension on the interface are included in the model.  Mathematically, this entails replacing the dynamic condition \eqref{dynamic on interface} with 
\be
\label{capillary dynamic}
	P_2 - P_1 = \sigma \frac{(1+\eta_y^2) \eta_{xx} - 2 \eta_x \eta_y \eta_{xy} + (1+\eta_x^2) \eta_{yy}}{ (1+\eta_x^2+\eta_y^2)^{3/2}} \qquad \textrm{on } z = \eta(x,y),
\ee
where $\sigma > 0$ is the coefficient of surface tension.  The right-hand side above is the mean curvature of the free boundary, and hence \eqref{capillary dynamic} enforces the Young--Laplace law for the pressure jump.

Thanks to Theorem~\ref{thm vorticity} and Lemma~\ref{lem y indep}, a straightforward adaptation of the proof of \cite[Theorem 2]{wahlen2014non} quickly yields the following result on the nonexistence of constant vorticity internal capillary-gravity waves.

\begin{theorem}[Capillary-gravity waves] \label{capillary theorem}
Any solution to the internal capillary-gravity wave problem \eqref{momentum equation}--\eqref{kinematic on interface}, \eqref{kinematic on rigid}, \eqref{capillary dynamic} satisfying 
\begin{enumerate}[label=\rm(\roman*)]
\item  $\boldsymbol{\omega}_1$ and $\boldsymbol{\omega}_2$ are constant, parallel, and nonzero,
\item  $\| \boldsymbol{u}_1 \|_{C^1}$, $\| \boldsymbol{u}_2 \|_{C^1}$, $\| \eta \|_{C^2} < \infty$, and
\item $\sup{(P_{1z} - P_{2z})|_{z = \eta}} < 0$,
\end{enumerate}
is necessarily two dimensional.  
\end{theorem}

Notice that the sign requirement on $P_{1z} - P_{2z}$ along the interface is consistent with the two-fluid Rayleigh--Taylor criterion due to Lannes \cite{lannes2013stability}, though it is not equivalent to well-posedness like in the one-fluid case.

\subsection*{Non-parallel vorticities}
Second, it is natural to ask whether Theorem~\ref{rigidity theorem} can be extended to the case $\boldsymbol{\omega}_1$ and $\boldsymbol{\omega}_2$ are non-parallel.  For instance, suppose that they are orthogonal with $\boldsymbol{\omega}_1$ aligned along the $x$-axis and $\boldsymbol{\omega}_2$ aligned along the $y$-axis.   In view of Lemma~\ref{lem y indep}, this would imply that 
\[
\begin{aligned}
	\boldsymbol{u}_1 & = \left(\tilde u_1(x,z), V_1, \tilde w_1(x,z)\right) &\qquad& \boldsymbol{u}_2 = \left( U_2, \tilde v_2(y,z), \tilde w_2(y,z) \right), \\
	P_1 & = \tilde P_1(x,z) & \qquad & P_2 = \tilde P_2(y,z),
\end{aligned}
\]
for constants $V_1$ and $U_2$.  We conjecture that this is not possible if $\rho_1 \neq \rho_2$, and even in the Boussinesq setting it can only be that the flow in both layers is shear --- that is, $\nabla\eta$, $\tilde w_1$, and $\tilde w_2$ vanish identically, while $\tilde u_1$ and $\tilde v_2$ are independent of the horizontal variables.  Indeed, the dynamic boundary condition on the interface would then give
\[
	\tilde P_1(x,\eta(x,y)) = \tilde P_2(y,\eta(x,y)) \qquad \textrm{for all } (x,y) \in \mathbb{R}^2,
\]
which coupled with the kinematic conditions appears to be overdetermined.  However, the argument for the parallel vorticity case do not apply directly, as we cannot project into a common two-dimensional domain.   

\subsection*{Pressure reconstruction}
Lastly, in the proof of Theorem~\ref{rigidity theorem}, we were confronted with the possibility that on some open subset $\mathcal{U} \subset \mathbb{R}^2$, there are two solutions to the incompressible steady Euler equations with (potentially different) constant densities and vorticities.  That is, the elliptic problem
\be
\label{pressure problem}
	\left\{
	\begin{aligned}
		\Delta \psi + \beta & = 0 \\
		\nabla \left( \frac{1}{2} |\nabla \psi|^2 - \beta \psi + g z + \frac{1}{\rho} P \right) & = 0
	\end{aligned}
	\right.
	\qquad \textrm{in } \mathcal{U}.
\ee
was satisfied by the triples $(\psi_1, \rho_1, \beta_1)$ and $(\psi_2, \rho_2, \beta_2)$.  In the context of the proof of Theorem~\ref{rigidity theorem}, we had additional information about the level sets of $\psi_1$ and $\psi_2$ due to the kinematic condition (for the three-dimensional problem), which was how we ultimately found that this situation could not occur unless $\rho_1 = \rho_2$, $\beta_1 = \beta_2$, and $\psi_1$ was an affine function of $\psi_2$.  However, one could reasonably ask whether the same conclusion follows simply from \eqref{pressure problem} if say $\psi_1$ and $\psi_2$ share a common streamline.  This question is of considerable independent interest, both mathematically and to hydrodynamical applications.  On the one hand,  \eqref{pressure problem} is a parameter-dependent Poisson problem coupled with an unusual gradient constraint.  Thus unique solvability falls into the broader category of unique continuation of elliptic PDE.  On the other hand, determining $(\psi, \rho,\beta)$ from \eqref{pressure problem} amounts to recovering the flow from pressure data, which has been the subject of a number of papers in the applied literature.  Constantin \cite{constantin2012pressure} provided an explicit formula for the surface elevation of a two-dimensional irrotational solitary wave in finite-depth water in terms of the trace of the pressure on bed.  The central observation of that work is that one can derive from the pressure on the bed and Bernoulli's principle Cauchy data for an elliptic equation describing the flow.  Henry \cite{henry2013pressure} extended this idea to general real-analytic vorticity (assuming the absence of stagnation points) using the Dubreil-Jacotin formulation of the steady water wave problem and Cauchy--Kovalevskaya theory.   Chen and Walsh \cite{chen2018unique} later  proved an analogous result with vorticity of Sobolev regularity and allowed for density stratification using strong unique continuation techniques.  See also \cite{clamond2013recovery,henry2018prediction,clamond2020extreme} for further results of this variety.   Pressure recovery for \eqref{pressure problem} is simpler in that we require constant vorticity and have pressure data on an open set, rather than the boundary.  However, it is important that we do not specify a priori the values of $\rho$ or $\beta$, which is a large departure from these earlier works.

\section*{Acknowledgments}
The research of RMC is supported in part by the NSF through DMS-1907584. The research of LF is supported in part by the NSF of Henan Province of China through Grant No. 222300420478 and the NSF of Henan Normal University through Grant No. 2021PL04.  The research of SW is supported in part by the NSF through DMS-1812436. 

\bibliographystyle{siam}
\bibliography{nonexistence}

\end{document}